\documentclass{amsart}

\usepackage[T1]{fontenc}
\usepackage[utf8]{inputenc}
\usepackage[USenglish]{babel}
\usepackage{textcase}

\usepackage[
	bookmarks=true,
	plainpages=false,
	linktocpage,
	colorlinks=true,
	citecolor=green!80!black,
	linkcolor=red!70!black,
	filecolor=magenta,
	urlcolor=magenta,
	breaklinks,
	pdfauthor={Martin Winter},
]{hyperref}

\usepackage{amsmath,amsthm}
\usepackage{calc, mathtools} 

\iftrue
\usepackage{amssymb} 
\else
\usepackage[charter,cal=cmcal]{mathdesign}
\fi

\usepackage{colortbl,color} 
\usepackage[dvipsnames]{xcolor}
\usepackage{centernot} 
\usepackage{array} 
\usepackage{enumitem,moreenum} 
\usepackage{cite} 
\usepackage[nameinlink,capitalize,noabbrev]{cleveref}
\usepackage{nicefrac}

\usepackage{tikz-cd} 

\usepackage[font=small,labelfont=bf]{caption}

\usepackage{blkarray}

\usepackage{inconsolata}

\newcommand{\ZZ}{\mathbb{Z}}    
                   
\def\^#1{^{(#1)}}
\def\s^#1{^{\smash{(#1)}}}

\newcommand{\wideword}[1]{\quad\text{#1}\quad}
\newcommand{\wideand}{\wideword{and}}

\def\:{\colon}

\newcommand{\cupdot}{\mathbin{\mathaccent\cdot\cup}}

\newcommand{\labelstyle}[1]{\upshape(\textit{#1})}
\newcommand{\mylabel}{\labelstyle{\roman*}}
\newenvironment{myenumeratearg}[1][\unskip]
    {\begin{enumerate}[label=\mylabel,#1]}
    {\end{enumerate}}

\def\itm#1{{\labelstyle{\romannumeral#1\relax}}}

\def\nlspace{\nolinebreak\space}

\newcommand{\freespace}{\kern.07em}

\newcommand{\ul}[1]{\underline{\smash{#1}}}

\newcommand{\msays}[1]{{\footnotesize\textcolor{red}{\textbf{M:} #1}}}
\newcommand{\TODO}{\msays{TODO}}

\newtheoremstyle{mythmstyle} 
    {\parsep}                    
    {\parsep}                    
    {\itshape}                   
    {}                           
    {\bfseries\scshape}          
    {.}                          
    {.5em}                       
    {}  
\newtheoremstyle{mydefstyle} 
    {\parsep}                    
    {\parsep}                    
    {}                   
    {}                           
    {\mdseries\scshape}          
    {.}                          
    {.5em}                       
    {}  

\numberwithin{equation}{section}

\theoremstyle{theorem}
\newtheorem{theorem}{Theorem}[section]
\newtheorem{corollary}[theorem]{Corollary}

\newtheorem{proposition}[theorem]{Proposition}

\theoremstyle{definition}

\newtheorem{remark}[theorem]{Remark}

\newtheorem{observation}[theorem]{Observation}

\crefname{theorem}{Theorem}{Theorems}
\crefname{proposition}{Proposition}{Propositions}
\crefname{lemma}{Lemma}{Lemmas}
\crefname{corollary}{Corollary}{Corollaries}
\crefname{remark}{Remark}{Remarks}
\crefname{example}{Example}{Examples}
\crefname{definition}{Definition}{Definitions}
\crefname{problem}{Problem}{Problems}
\crefname{observation}{Observation}{Observation}
\crefname{construction}{Construction}{Construction}

\DeclareMathOperator{\im}{im}

\DeclareMathOperator{\Sym}{Sym}

\DeclareMathOperator{\Hex}{Hex}

\let\eset=\varnothing

\let\<=\langle
\let\>=\rangle

\def\...{...}
\newcommand{\shortStyle}{\textit}
\newcommand{\ie}{\shortStyle{i.e.,}}

\newcommand{\eg}{\shortStyle{e.g.}}

\newcommand{\cf}{\shortStyle{cf.}}

\newcommand{\resp}{resp.}

\makeatletter
\renewcommand*{\eqref}[1]{%
  \hyperref[{#1}]{\textup{\tagform@{\ref*{#1}}}}%
}
\makeatother

\def\nlspace{\nolinebreak\space}
\def\nls{\nlspace}

\def\tttcube{$2\times2\times2$ cube}
\def\tttcubes{$2\times2\times2$ cubes}

\def\Z#1{G_{#1}}
\def\Zd#1{G_{#1}^*}

\begin{document}


\expandafter\title
[The clique graphs of the hexagonal lattice]
{The clique graphs of the hexagonal lattice -- an explicit construction and a short proof of divergence}
		
\author[M. Winter]{Martin Winter}
\address{Mathematics Institute, University of Warwick, Coventry CV4 7AL, United Kingdom}
\email{martin.h.winter@warwick.ac.uk}
	






		
\date{\today}
\begin{abstract}
We present a new, explicit and very geometric construction for~the iterated clique graphs of the hexagonal lattice $\Hex$ which makes apparent its clique-divergence and sheds light on some previous observations, such as the boundedness of the degrees and clique sizes of $k^n \Hex$ as $n\to\infty$.
\end{abstract}

\maketitle

\section{Introduction}

Given a (potentially infinite) simple graph $G$, its \emph{clique graph} $kG$ is the intersec\-tion graph of the cliques in $G$. More precisely, $kG$ has as its vertices the \emph{cliques} of $G$ (\ie~the inclusion maximal complete subgraphs), two of which are adjacent in $kG$ if they have non-empty intersection in $G$. A graph is said to be \emph{clique divergent} if its \emph{iterated clique graphs} $kG, k^2 G,k^3 G,...$ are pairwise non-isomorphic. It is called \emph{clique convergent} otherwise.

The \emph{hexagonal lattice} $\Hex$ (shown in \cref{fig:hexagonal_lattice}) is known to be clique divergent.
This was first hinted to by the findings of \cite{larrion1999clique,larrion2000locally}, which proved the clique divergence of its finite quotient graphs (6-regular triangulations of the torus and Klein bottle).
%
Later, divergence was proven directly in \cite{limbach2023characterising}, building on an explicitly construction of the clique graphs of $\Hex$ (and of other graphs) introduced in \cite{baumeister2022clique}.

\begin{figure}[h!]
    \centering
    \includegraphics[width=0.365\textwidth]{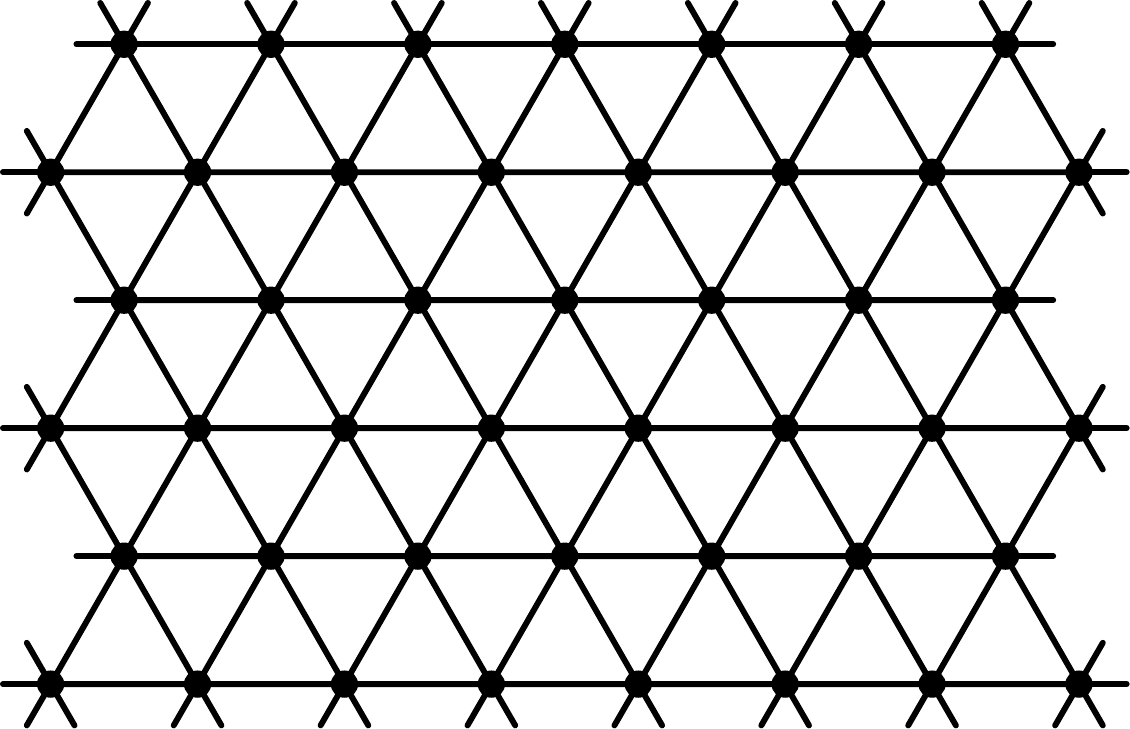}
    \caption{The hexagonal lattice.}
    \label{fig:hexagonal_lattice}
\end{figure}

The article at hand presents a new explicit, and as we find, rather neat construction of the clique graphs of the hexagonal lattice, that makes its clique-divergence completely apparent.
Our construction is noteworthy in that, once the idea~is~presented, the proofs require little more than some 3-dimensional intuition.
Moreover, this new perspective sheds light on several previous observations, such as the boundedness of the degrees and clique sizes of $k^n \Hex$ as $n\to\infty$.

Even though the result applies to a single object only, we do believe that it~is~of interest: the hexagonal lattice itself and its quotients have received notable attention in the literature on clique dynamics, some results of which~we~were able~to~re\-produced using compacter arguments.


In \cref{sec:construction} we introduce the construction and the main result, which is proven in \cref{sec:proof}. 
In \cref{sec:comments} we make some more observations regarding the construction and explain further relation to the literature.


\section{Construction and statement of main result}
\label{sec:construction}

In the following let $\Z d$ denote the \emph{$\ell^\infty$-unit distance graph} of the $\ZZ^d$ lattice.\nlspace
That is, ${\Z d}$ has vertex set $\ZZ^d$, with $x,y\in\ZZ^d$ being adjacent in ${\Z d}$ if and only~if their~$\ell^\infty$-distance equals $1$, that is, if
$$\|x-y\|_{\ell^\infty}:=\max_i|x_i-y_i|=1.$$
%

To establish our main result about the hexagonal lattice it is completely sufficient to restrict to $d=3$, on which we shall focus in the following.
The general~definition is however still useful: the case $d=2$ is especially suited for visualizations~that~provide intuition (\eg\ see \cref{fig:G2}).
Moreover, our result generalizes in some form to $d\in\{1,2,3\}$, fails however for $d\ge 4$. We discuss this further in \cref{sec:special_values}.



\begin{figure}[h!]
    \centering
    \includegraphics[width=0.23\textwidth]{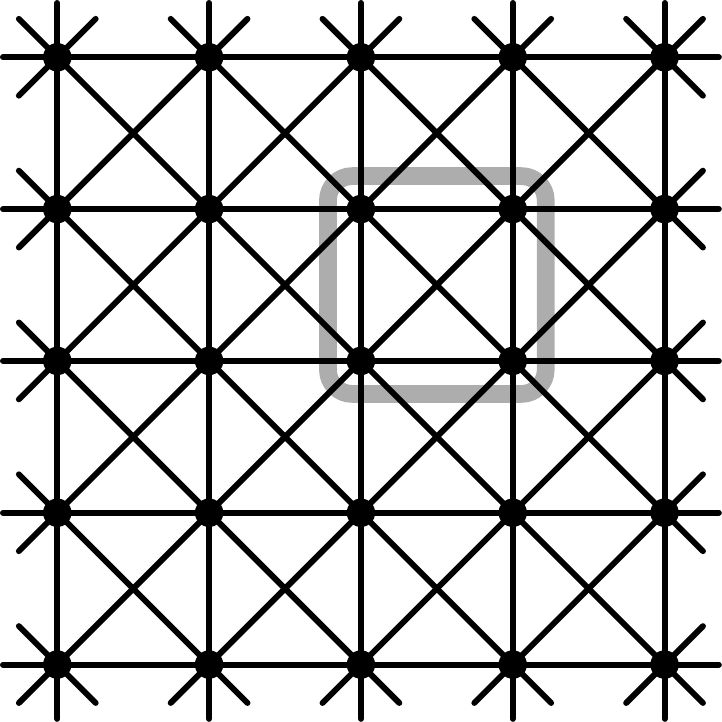}
    \caption{The graph $\Z 2$ with a highlighted clique (a ``$2\times 2$ square'').\nls All cliques are of this form.}
    \label{fig:G2}
\end{figure}

The relevance of $d=3$ is as follows:
the hexagonal lattice can be obtained as~the subgraph of ${\Z 3}$ induced on points with coordinate sum zero:
$$\Hex:={\Z 3}\big[x\in\ZZ^3\,\big\vert \, x_1+x_2+x_3=0 \big].$$
The points with a given coordinate sum we shall call a \emph{layer} of $G_d$.
%
%
Our main~observation is then that \emph{all} (even) clique graphs $k^{2n}\Hex$ can be interpreted as subgraphs of ${\Z 3}$ induced on one or more such layers. 

For general $d\ge 1$ and $n\ge 0$ we introduce the \emph{layered graph}
$${\Z d}(n):={\Z d}\Big[x\in\ZZ^d\,\Big\vert \,\big|\!\textstyle\sum_i x_i\big|\le n\Big].$$

\begin{figure}[h!]
    \centering
    \includegraphics[width=0.63\textwidth]{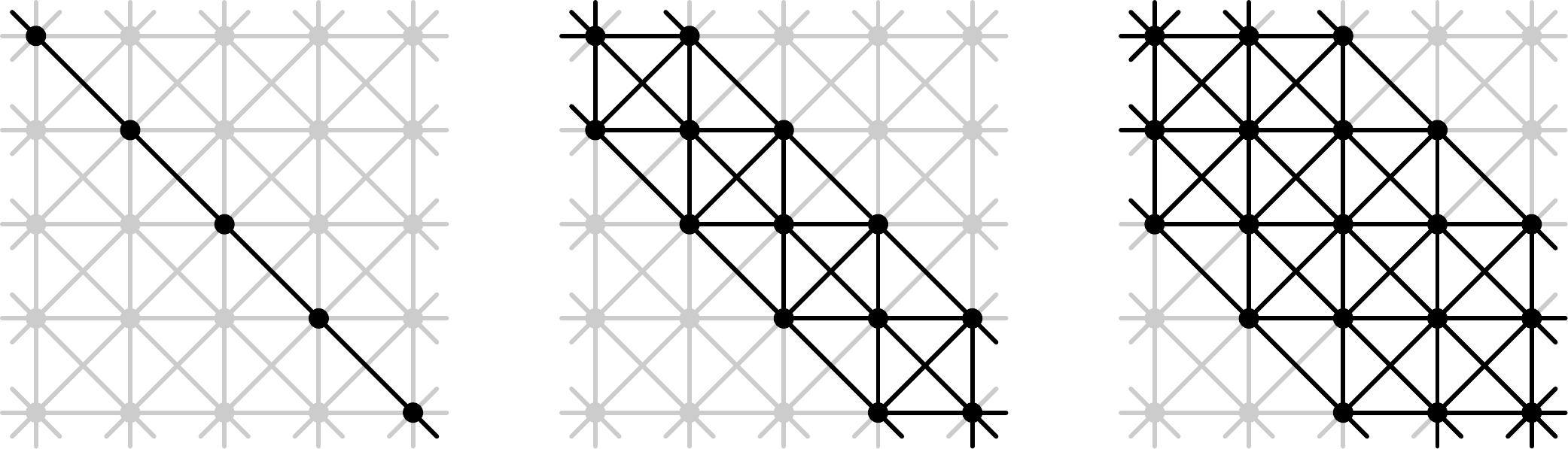}
    \caption{The layered graphs $G_2(0), G_2(1)$ and $G_2(3)$.}
    \label{fig:layered}
\end{figure}

\noindent
In particular, $\Hex= G_3(0)$.
Our core result for the hexagonal lattice reads 
$$k^{2n} \Hex \cong \Z 3(n),$$
from which clique-divergence is apparent (see also \cref{sec:divergence}).

To also state the result for odd clique graphs $k^{2n+1}\Hex$, we introduce the~``dual'' graph $\Zd d$, \ie\ the $\ell^\infty$-unit distance graph of the half-integer lattice \mbox{$\ZZ^d+\nicefrac12$} (that is, all coordinates are half-integers).
It~is clearly isomorphic to $\Z d$. 
The corresponding layered graph $\Zd d(n)$ is defined analogously:
$${\Zd d}(n):={\Zd d}\Big[x\in\ZZ^d+\nicefrac12\,\Big\vert \,\big|\!\textstyle\sum_i x_i\big|\le n\Big].$$

The main result can now be stated in full:

\begin{theorem}\label{res:main}
There are natural isomorphisms
$$k{\Z 3}(n) \cong {\Zd 3}(n+\nicefrac12) \wideand
k{\Zd 3}(n) \cong {\Z 3}(n+\nicefrac12).$$
Combining these yields $k^2 {\Z 3}(n) \cong {\Z 3}(n+1)$, and in particular,
%
$$
k^n \Hex \cong
\begin{cases}
    G_3\big(\nicefrac n2\big) & \text{if $n$ is even} \\
    G_3^*\big(\nicefrac n2\big) & \text{if $n$ is odd}
\end{cases}.
$$
%
\end{theorem}

\section{Proof of main result}
\label{sec:proof}

We believe that, once stated, verifying \cref{res:main} is fairly straightforward.\nls
The proof below will contain no real surprises.
It does however require us to verify some technical points that are best dealt with using some 3-dimensional intuition. 

We first present the main argument as a sequence of simple observations. Many of them are at least plausible from ``visual inspection''. For some of them we provide~more detailed arguments further below:
%

%
\begin{myenumeratearg}[itemsep=0.5em, topsep=0.3em]
    \item 
    The cliques of $G_3$ are exactly the ``\tttcubes'' in $G_3$, that is, they are~of the form $x+\{0,1\}^3$ with $x\in\ZZ^3$
    (\cref{fig:G2} shows the analogue situation in $G_2$, where the cliques are ``$2\times 2$ squares'').
    
    
    \item 
    A cube $x+\{0,1\}^3$ in $G_3$ has its centroid at $x+\{\nicefrac 12\}^3$, which is a vertex of $G_3^*$. 
    This correspondence yields an isomorphism $k G_3\cong G_3^*$.

    \item
    Since $G_3(n)$ is an induced subgraph of $G_3$, each clique $Q$ in $G_3(n)$ extends to a clique $\tilde Q$ in $G_3$, that is, $Q=\tilde Q\cap G_3(n)$.
    In fact, more is true:
    \begin{enumerate}[label=\alph*.,topsep=0.3em]
        \item the extension $\tilde Q$ is unique.
        \item cliques $Q_1, Q_2$ in $G_3(n)$ intersect if and only if their extensions $\tilde Q_1, \tilde Q_2$ intersect in $G_3$.
    \end{enumerate}
    We will provide justification for a.\ and b.\ below.
\end{myenumeratearg}


The discussion so far allows us to define a graph embedding
$$
\iota\: kG_3(n) \hookrightarrow k G_3 \overset{\sim}\to G_3^*, 
\quad 
Q 
\overset{\text{\itm3 a.}}\longmapsto 
\tilde Q 
\overset{\text{\itm1}}= 
x+\{0,1\}^3 
\overset{\text{\itm2}}\longmapsto 
x+\{\nicefrac12\}^3,
$$
and we can consider $k G_3(n)$ as a subgraph $kG_3(n)\cong\im(\iota)\subseteq G_3^*$.
By~\itm3 b.~we~can consider $kG_3(n)$ even as~an \emph{induced} subgraph of $G_3^*$.
It therefore remains to determine the vertices of $G_3^*$ in the image of $\iota$.
We need two more observations:
%
%
\begin{myenumeratearg}[itemsep=0.5em, topsep=0.3em]
\setcounter{enumi}{3}  
    \item A clique $\tilde Q$ in $G_3$ is an extension of a clique in $G_3(n)$ if and only if~$\tilde Q$~inter\-sects $G_3(n)$ in \emph{at least two} vertices (see \cref{fig:one_step} for the analogous situation within $G_2$)
    \item
    A \tttcube\ $x+\{0,1\}^3$ intersects $G_3(n)$ in at least two vertices if and only if its centroid $x+\{\nicefrac12\}^3$ has coordinate sum of absolute value $\le n+\nicefrac12$.

\end{myenumeratearg}
\begin{figure}[h!]
    \centering
    \includegraphics[width=0.75\textwidth]{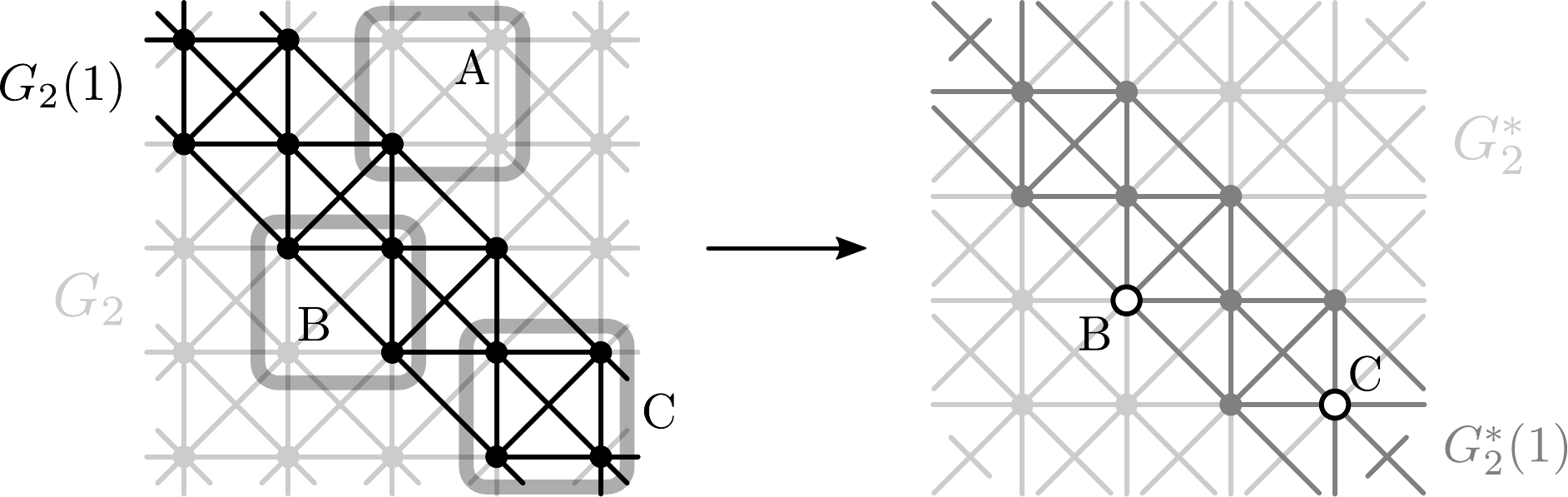}
    \caption{The ``squares'' A, B and C intersect $G_2(1)$, but only B and C intersect in more than one vertex and therefore correspond to cliques in $G_2(1)$. Those squares then yield vertices of $k G_2(1)$.
    Note that $k G_2(1)\cong G_2^*(1)$ in contrast to $d=3$, where $kG_3(1)\cong G_3^*(1+\nicefrac12)$.}
    \label{fig:one_step}
\end{figure}
This shows that $kG_3(n)\cong \im(\iota) = G_3^*(n+\nicefrac12)$ and concludes this part of the proof.
By swapping $G_3$ and $G_3^*$ we obtain an analogous proof for the other isomorphism.

We now provide arguments for \itm3 a., b., as well as \itm4 and \itm5.








\subsection*{Claim \itm3, a}
\emph{Every clique $Q$ in $G_3(n)$ has a unique extension $\tilde Q$ in $G_3$.}

Suppose that there are two distinct cliques (aka.\ cubes) $\tilde Q,\tilde Q'$ in $G_3$ that extend the clique $Q$ of $G_3(n)$.
Then $Q\subseteq \tilde Q\cap\tilde Q'$, which is a shared face of the two cubes.
However, considering \cref{fig:G2} we see that an intersections of a \tttcube\ with $G_3(n)$ that has at least two vertices (such as $Q$) never lies completely inside a face.

\begin{figure}[h!]
    \centering
    \includegraphics[width=0.45\textwidth]{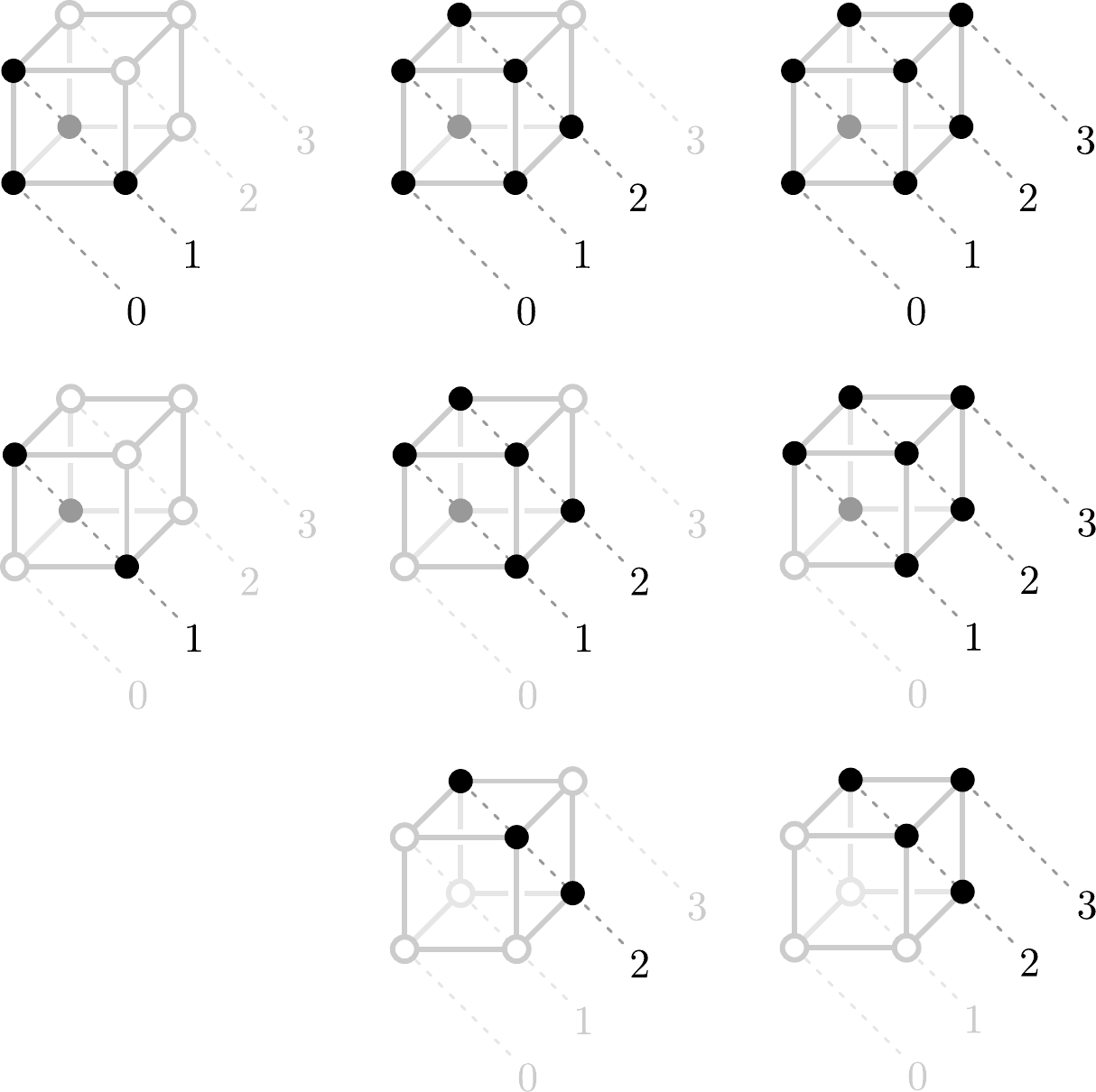}
    \caption{The seven ways in which a \tttcube\ can intersect $G_3(n)$ in at least two vertices. The dashed lines represent the layers of $G_3$ that intersect the cube (the numbers are for reference and do not necessarily indicate the coordinate sum of the layer). The line is black if the layer is in $G_3(n)$.}
    \label{fig:3_cube_variants}
\end{figure}

\subsection*{Claim \itm3, b}
\emph{Cliques $Q_1, Q_2$ in $G_3(n)$ intersect in $G_3(n)$ if and only if their~extension $\tilde Q_1, \tilde Q_2$ intersect in $G_3$.}

One direction is obvious.
For the other direction consider \cref{fig:cube_intersection_variants}:
it shows the five way in which two distinct \tttcubes\ in $G_3$ can intersect (up to symmetry of $G_3(n)$).
The figure also highlights layers of $G_3(n)$ that must necessarily intersect the cubes in order for $G_3(n)$ to intersect each cube in at least two vertices.
It~is~evident from the figure that these intersections necessarily contain vertices that lie in both cubes.
In other words, these cubes also intersect when restricted to $G_3(n)$.



\begin{figure}[h!]
    \centering
    \includegraphics[width=0.9\textwidth]{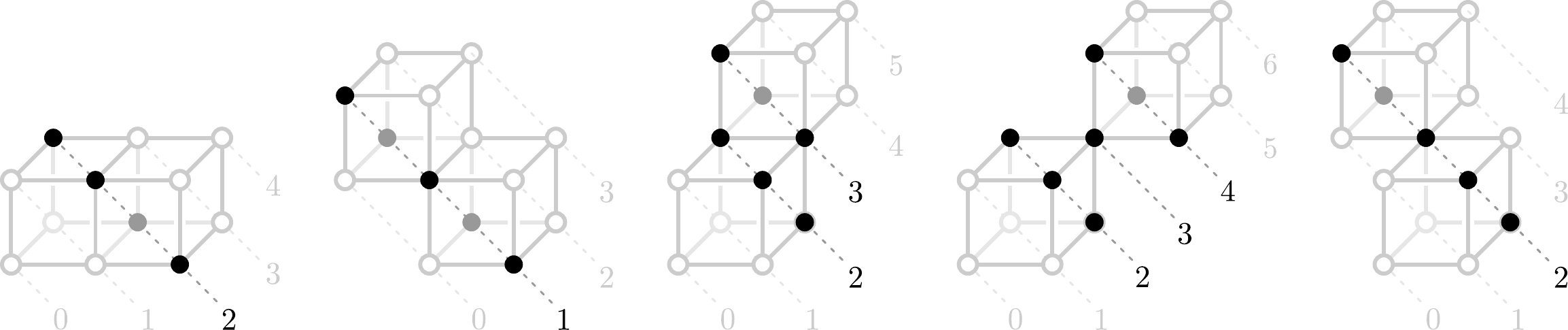}
    \caption{The five configurations (up to symmetries of $G_3(n)$) in which two distinct \tttcubes\ in $G_3$ can intersect.
    Each figure highlights the minimal amount of consecutive layers that intersects each cube in~at least two vertices 
    (strictly speaking, in the second case from the left it must be at least one of the layers 1 and 2).}
    \label{fig:cube_intersection_variants}
\end{figure}

\subsection*{Claim \itm4}
\emph{%
A clique $\tilde Q$ in $G_3$ is an extension of a clique in $G_3(n)$ if and only~if~$\tilde Q$ inter\-sects $G_3(n)$ in at least two vertices.
}

If $\tilde Q\cap G_3(n)$ has a single vertex, then it is not a clique of $G_3(n)$, since $G_3(n)$ has no isolated vertices.
Conversely, suppose $\tilde Q$ intersects $G_3(n)$ in at least two~ver- tices.
Let $Q$ be a clique of $G_3(n)$ that contains $\tilde Q\cap G_3(n)$, and let $\tilde Q'$~be~its~ex\-tension. 
If $\tilde Q\not=\tilde Q'$ then $\tilde Q\cap G_3(n)\subseteq \tilde Q\cap\tilde Q'$ must be a shared face of the cubes.
But considering once more the possible ways in which a cube can intersect $G_3(n)$ in at least two vertices in \cref{fig:3_cube_variants}, we see that this is not possible. Thus $\tilde Q\cap G_3(n)=\tilde Q'\cap G_3(n)=Q$ is a clique.




\subsection*{Claim \itm5}
\emph{%
A cube $\tilde Q:= x+\{0,1\}^3$ intersects $G_3(n)$ in at least two vertices if and only if its centroid $x+\{\nicefrac12\}^3$ has coordinate sum of absolute value $\le n+\nicefrac12$.
}

Let $s$ be the coordinate sum of $x$.
The layers of $G_3$ that intersect $\tilde Q$~in~at~least two vertices have coordinate sum $s+1$ and $s+2$. 
Thus, for $G_3(n)$ to intersect $\tilde Q$ we require $|s+1|\le n$ or $|s+2|\le n$. Elementary computation shows that this is equivalent to $|s+\nicefrac32|\le n+\nicefrac12$.
Since $s+\nicefrac 32$ is the coordinate sum of the centroid of $\tilde Q$, the claim follows.

\section{Further comments}
\label{sec:comments}

\subsection{Bounded degree and clique number}
\label{sec:bounded}

The vertex degree of ${\Z 3}$ is $3^3-1=26$. 
As we have seen, $k^n\Hex$ appears as a subgraph of $G_3$, which shows~that~the~vertex degrees in $k^n \Hex$ stay bounded as $n\to\infty$.
The number 26 also played a major role in the proofs of \cite{limbach2023characterising}, where it was considered a curiosity.
Our constructions~provides an explanation for the appearance of this peculiar number. 

We also note that the clique number of ${\Z 3}$ is $8$. 
This fact gives a concise explanation for the observation that clique numbers of $k^n\Hex$ stay bounded as $n\to\infty$. 
This has previously been observed for the finite quotients of $\Hex$ in \cite{larrion1999clique,larrion2000locally}.

\subsection{Clique-divergence of $\boldsymbol{\mathrm{Hex}}$}
\label{sec:divergence}

The explicit form $k^{2n}\Hex\cong G_3(n)$ makes apparent the clique-divergence of $\Hex$. Here is a more explicit argument:
consider the subgraph $H_n$ of $G_3(n)$ induced on all vertices of degree $<26$. One can show that~for $n$ sufficiently large $H_n$ has two connected components, and the graph-theoretic~distance between those components diverges as $n\to\infty$.



\subsection{Clique-convergence of $\boldsymbol{G_d}$}

As noted in \cref{sec:proof} \itm2, we have $k G_3\cong G_3^*$, and by symmetry, $kG_3^*\cong G_3$.
Thus $k^2 G_3\cong G_3$, and $G_3$ is clique-convergent.

The argument applies completely analogous to $G_d$ for general $d\ge 1$: observe first that the cliques in $G_d$ are exactly the $2\times\cdots\times 2$ cubes $x+\{0,1\}^d$. 
The isomorphism $kG_d\cong G_d^*$ is then given by $x+\{0,1\}^d\mapsto x+\{\nicefrac12\}^d$.

For example, for $d=1$ we have that $G_1$ is the infinite path graph, which indeed is clique-convergent.

\subsection{Other values for $\boldsymbol{d}$}
\label{sec:special_values}

In this article we have been motivated mainly by the clique dynamics of the hexagonal lattice, and therefore, the case $d=3$. 
As it turns out, the statement of \cref{res:main} and its proof given in \cref{sec:proof} can be easily adjusted to also work with a few other values of $d$, thought $d=3$ remains the most interesting one of them:

\begin{theorem}\label{res:new_main}
If $d\in\{1,2,3\}$ and $n\ge 0$, but $(d,n)\not=(1,0)$, then there are~natural iso\-morphisms
$$k{\Z d}(n) \cong {\Zd d}(n+d/2-1) \wideand
k{\Zd d}(n) \cong {\Z d}(n+d/2-1).$$
Combining these yields $k^2 {\Z d}(n) \cong {\Z d}(n+d-2)$. 
%
\end{theorem}

Let us consider the values $d\in\{1,2\}$ in some more detail, and also explain where the proof fails for $d\ge 4$.

For $\boldsymbol{d=1}$ the graph $G_1(0)$ is a single vertex and must be excluded from \cref{res:new_main} (in this case the proof in \cref{sec:proof} fails in step \itm4, where we require that $G_1(0)$ has no isolated vertices).
For general $n\ge 1$, the graph $G_1(n)$ is a path of length $2n+1$. 
The peculiarity of the case $d=1$ is that $k^2 G_1(n)\cong G_1(n-1)$, that is, the clique graphs are shrinking, completely in agreement with what we expect from the finite path graph.

For $\boldsymbol{d=2}$ we have $k^2G_2(n)\cong G_2(n)$, and so the clique graphs are ``stable''. A special case is $G_2(0)$, which is the infinite path. See \cref{fig:layered} for other examples. 

For $\boldsymbol{d\ge 4}$ there is no direct analogue of \cref{res:new_main}.
The proof of \cref{sec:proof} fails in step \itm3 b.: two cliques $Q_1,Q_2$ in $G_4(n)$ can be disjoint, while their extensions in $G_4$ intersect. Here is an example:
%
%
the $2\times2\times2\times 2$ cubes
$$(1,-1,0,0)+\{0,1\}^4\wideand (0,0,1,-1)+\{0,1\}^4$$
intersect only in the point $(1,0,1,0)$. Yet their intersections with $G_4(1)$ are disjoint, even though each cube intersects $G_4(1)$ in at least two vertices.
We can however still find $kG_d(n)$ as a spanning subgraph of $G_d^*(n+d/2-1)$, and vice versa.

\subsection{Triangulations of the torus and the Klein bottle}
\label{sec:torus_Klein_bottle}

Any group action~$\Gamma\curvearrowright\Hex$ extends uniquely to actions $\Gamma\curvearrowright {\Z 3}$ and $\Gamma\curvearrowright G_3(n)$ that preserves coordinate sums.
Taking the quotient of $G_3(n)$ by such an action yields an explicit~description for the clique graphs of the quotient $T:=\Hex/\Gamma$, which is a 6-regular triangulation of an unbounded surface (\ie\ the torus, the Klein bottle, the infinite cylinder, the infinite Möbius strip or the plane):
$$k^{2n}T=k^{2n}(\Hex/\Gamma) \cong (k^{2n}\Hex)/\Gamma \cong  G_3(n)/\Gamma.$$
Some more technicalities are involved in verifying the two isomorphism (see also \cite[Lemma 4.4]{limbach2023characterising}), but all in all, we obtain a concise description of the clique graphs first mentioned in \cite{larrion2000locally}, that also makes transparent their linear growth as $n\to\infty$.

\subsection{Relation to the geometric clique graph}
\label{sec:geometric_clique_graph}

The geometric clique graphs $\mathcal G_n$ (introduced in \cite{baumeister2022clique}) provides an alternative description for the clique graphs of the hexagonal lattice (and more generally, of all ``locally cyclic graph of minimum degree $\delta\ge 6$'').
The vertices of $\mathcal G_n$ are the triangular shaped subgraphs of $\Hex$ (shown in \cref{fig:pyramids}) of side length $m$, where $m\le n$ and $m\equiv n\pmod 2$, subject to a non-trivial set of rules for adjacency (see \cite[Definition 4.1]{baumeister2022clique} or \cite[Definition 2.1]{limbach2023characterising}).
It~was proven in \cite[Theorem 6.8 + Corallary 7.8]{baumeister2022clique} that $k^n \Hex \cong \mathcal G_n$.

\begin{figure}[h!]
    \centering
    \includegraphics[width=0.6\textwidth]{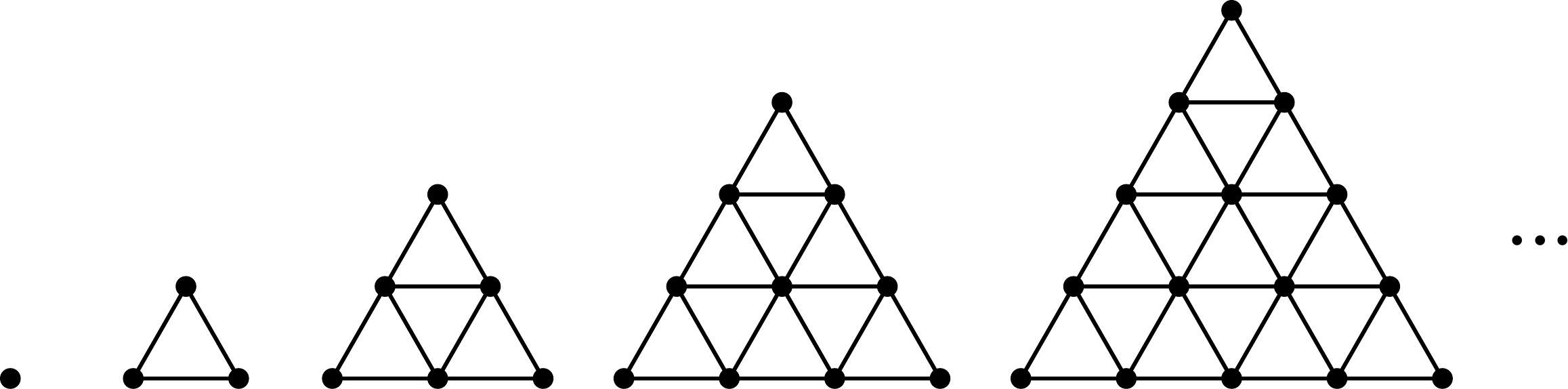}
    \caption{The ``triangular-shaped subgraphs'' of $\Hex$ of side length~0, ..., 4, as used in the construction of the geometric clique graph $\mathcal G_n$}
    \label{fig:pyramids}
\end{figure}

%

Our description of $k^n\Hex$ allows for an alternative interpretation of $\mathcal G_n$ and yields a natural explanation for the otherwise ad hoc adjacency rules.
Define the positive \resp\ negative orthant:
$$O^+:=\{x\in \ZZ^3\mid x_1,x_2,x_3\ge 0\}\wideand O^-:=\{x\in \ZZ^3\mid x_1,x_2,x_3\le 0\}$$
and set $O^\pm:=O^+\cup O^-$. To each vertex $x\in G_3\cupdot G_3^*$ (which is a point with integer or half-integer coordinates) we associate a triangular-shaped subgraph $T_x\subset \Hex=G_3(0)$ as follows (\cf\ \cref{fig:orthant_intersection}):
$$T\:\,x\;\longmapsto\; T_x:=G_3(0)\cap (2x+O^\pm).$$%
\begin{figure}[h!]
    \centering
    \includegraphics[width=0.35\textwidth]{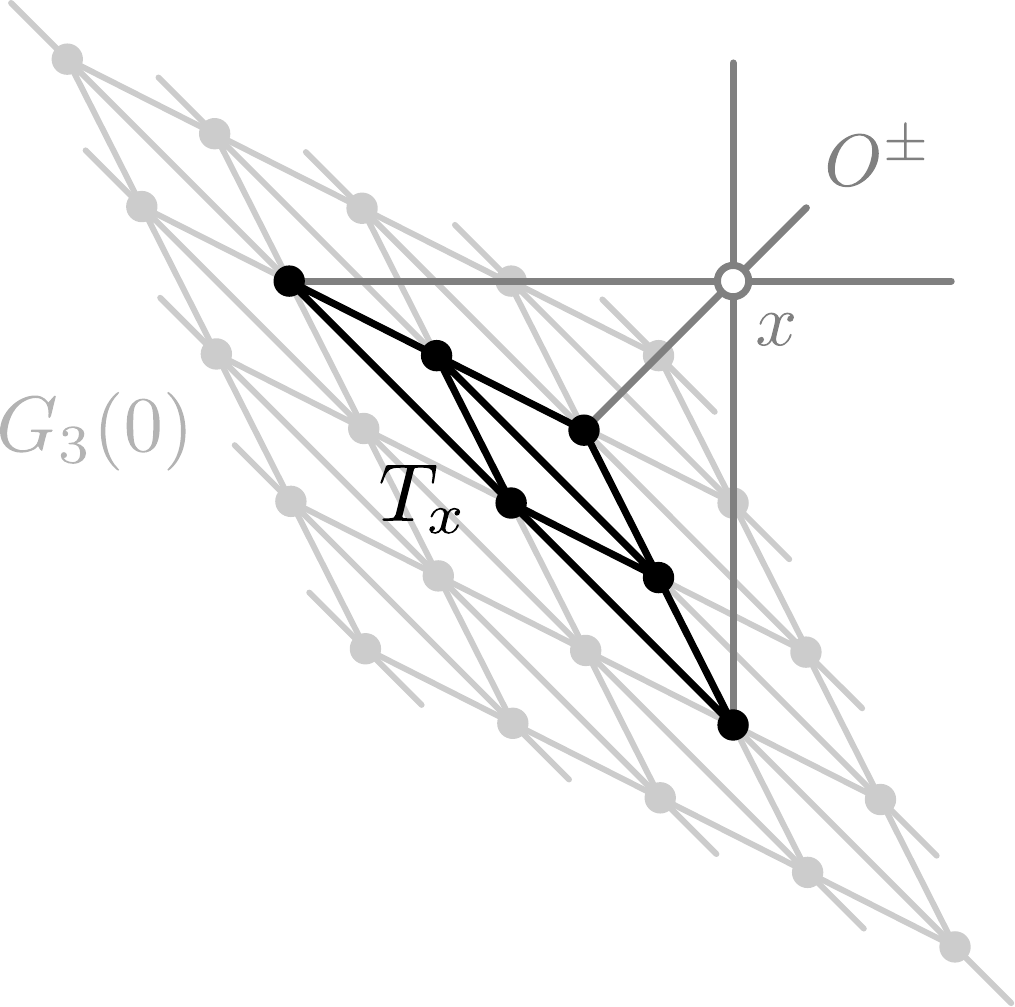}
    \caption{The intersection of $\Hex=G_3(0)$ with $x+O^\pm$, yields~a~tri\-angular-shaped subgraph $T_x$.}
    \label{fig:orthant_intersection}
\end{figure}%
This yields an interpretation for the vertices of $k^{2n}\Hex\cong G_3(n)$ resp.\ $k^{2n-1}\Hex\cong G_3^*(n)$ as triangular-shaped subgraphs of $\Hex$ which is in accordance with the interpretation from $\mathcal G_n$. In fact, $x,y\in G_3\cupdot G_3^*$ are adjacent if and only if $T_x$ and $T_y$ are adjacent in $\mathcal G_n$, providing a new interpretation for the adjacency rules in $\mathcal G_n$.

\bibliographystyle{abbrv}
\bibliography{literature}

\end{document}